\documentclass[11pt]{amsart}
\usepackage{amsmath,amssymb,amsthm}
\usepackage[margin=1.1in]{geometry}
\usepackage{enumitem}
\usepackage{booktabs}
\usepackage{tikz}
\usepackage[colorlinks=true,linkcolor=blue!50!black,citecolor=blue!50!black,urlcolor=blue!50!black]{hyperref}

\newtheorem{theorem}{Theorem}[section]
\newtheorem{lemma}[theorem]{Lemma}
\newtheorem{proposition}[theorem]{Proposition}

\newtheorem{thmx}{Theorem}

\theoremstyle{definition}
\newtheorem{definition}[theorem]{Definition}
\newtheorem{example}[theorem]{Example}
\newtheorem{question}[theorem]{Question}
\theoremstyle{remark}
\newtheorem{remark}[theorem]{Remark}

\newcommand{\Z}{\mathbb{Z}}
\newcommand{\R}{\mathbb{R}}
\newcommand{\PP}{\mathbb{P}}
\newcommand{\cone}{\operatorname{cone}}
\newcommand{\conv}{\operatorname{conv}}
\newcommand{\spn}{\operatorname{span}}
\newcommand{\ud}{\operatorname{ud}}
\newcommand{\udind}{\operatorname{ud}_{\mathrm{ind}}}
\newcommand{\V}{\mathcal{V}}
\newcommand{\E}{\mathcal{E}}
\newcommand{\pair}[2]{\langle #1, #2\rangle}
\newcommand{\AP}{\textup{(AP)}}

\title[Unimodular defect and antipodal lattice points]{Unimodular defect and antipodal lattice points \\ of smooth reflexive polytopes}
\author{Binnan Tu}


\begin{document}

\begin{abstract}
Let $P$ be a smooth reflexive polytope and let $Q=P^{*}$ be the dual of smooth Fano polytope.
Ewald's conjecture (1988) predicts that the lattice points of $P\cap(-P)$ contain a lattice basis.
In general it is not even known whether $P\cap(-P)$ contains a nonzero lattice point; we call this weakest form of the conjecture the \emph{antipodal condition}.
We define the \emph{unimodular defect} of a finite set of lattice vectors as the least number of vectors that must be removed to leave a unimodular set, and the unimodular defect of $Q$ as that of its vertex set.
Defect zero gives the unimodular smooth Fano polytopes, for which Ewald's conjecture is known.
Our main result is that every smooth Fano polytope of unimodular defect at most two has a \emph{flat ridge}, a ridge $W$ with $|\det(v,W)|\le 1$ for all vertices $v$; consequently its dual satisfies the antipodal condition.
We also consider removed sets of pairwise non-adjacent vertices.
Using results of Nill on independent vertices, we show that every such set has at most two elements unless $Q$ is unimodular, so the antipodal condition holds whenever such a set exists.
\end{abstract}

\maketitle

\section{Introduction}

\subsection{Ewald's conjecture and antipodal lattice points}

Let $N\cong\Z^n$ be a lattice, $M=\operatorname{Hom}(N,\Z)$ its dual lattice, and $\pair{\cdot}{\cdot}\colon N_\R\times M_\R\to\R$ the pairing, where $N_\R=N\otimes_\Z\R$ and $M_\R=M\otimes_\Z\R$.
A \emph{smooth Fano polytope} is a full-dimensional lattice polytope $Q\subset N_\R$ with the origin in its interior such that the vertices of every facet of $Q$ form a $\Z$-basis of $N$.
Its dual
\[
P=Q^{*}=\{m\in M_\R:\ \pair{v}{m}\ge -1\ \text{for all } v\in Q\}
\]
is a \emph{smooth reflexive polytope}, and every smooth reflexive polytope arises in this way.
Smooth Fano polytopes $Q$ correspond to smooth toric Fano varieties $X=X_Q$: the cones over the faces of $Q$ form the fan of $X$, and $P$ is the polytope of the anticanonical divisor $-K_X$ \cite{Batyrev99,CLS}.
In symplectic geometry, smooth reflexive polytopes are known as \emph{monotone polytopes}; they are the moment polytopes of monotone symplectic toric manifolds \cite{McDuff11,CPS}.
In each dimension there are finitely many of them up to unimodular equivalence, and they have been classified up to dimension $9$ \cite{Obro07,LP}.

We write $\V=\V(Q)$ for the vertex set of $Q$ and
\[
\E(P)=P\cap(-P)\cap M=\{m\in M:\ |\pair{v}{m}|\le 1\ \text{for all } v\in\V\}
\]
for the set of lattice points $m$ of $P$ whose antipode $-m$ also lies in $P$.
In 1988 Ewald \cite{Ewald88} asked whether every smooth Fano polytope is unimodularly equivalent to a lattice polytope contained in the cube $[-1,1]^n$.
In terms of the dual polytope this is the following statement, known as \emph{Ewald's conjecture}:
\begin{quote}
\emph{For every smooth reflexive polytope $P$, the set $\E(P)$ contains a $\Z$-basis of $M$.}
\end{quote}
The conjecture has been verified by computer for $n\le 5$ by Kreuzer and Nill \cite[Proposition~4.9]{KN09} and for $n\le 7$ by {\O}bro \cite[\S4.1]{ObroThesis}, who checked the stronger property that $\E(P)\cap F$ contains a lattice basis for every facet $F$ of $P$ (see \cite[Theorem~3.9 and Remark~3.10]{CPS}).
Proofs are known only for special classes.
Crespo, Pelayo and Santos \cite{CPS} proved the conjecture for \emph{deeply monotone} polytopes, the first broad class in arbitrary dimension, and studied its behaviour under fibre bundles.
The author \cite{Tu24} showed it for \emph{unimodular smooth Fano polytopes}, those whose vertex set is a unimodular configuration, and proved that every deeply monotone polytope is the dual of a unimodular smooth Fano polytope.

\begin{definition}[Antipodal condition]\label{def:AP}
A smooth reflexive polytope $P$ (or its dual $Q$) satisfies the \emph{antipodal condition} \AP{} if $\E(P)\neq\{0\}$, that is, if $P$ contains a pair of antipodal lattice points $\pm m\ne 0$.
\end{definition}

Payne says that Ewald's conjecture has been verified in dimension at most seven, but that ``it remains unknown in higher dimensions whether there exists a single nonzero lattice point'' in $P\cap(-P)$ \cite[Remark~4.6]{Payne09}.
McDuff makes the same observation \cite{McDuff11}, and Crespo, Pelayo and Santos repeat that ``it is not even known whether there is a monotone polytope with $\E(P)=\{0\}$'' \cite[Section~1]{CPS}; they also ask for the minimum of $|\E(P)|$ in each dimension and remark that it could be $1$ \cite[Section~6.2]{CPS}.
Since $\E(P)$ is centrally symmetric, \AP{} says exactly that $|\E(P)|\ge 3$.
By computer, \AP{} is known for $n\le9$: Brendel reports that {\O}bro (for $n\le8$) and Paffenholz (for $n=9$) checked that every facet of every smooth reflexive polytope of dimension $n\le9$ contains a nonzero point of $\E(P)$ \cite[\S3.4]{Brendel}.
Apart from the classes for which Ewald's conjecture is known, the only general result on \AP{} that we are aware of is that it passes from the fibre to the total space of a monotone bundle \cite[Proposition~5.8]{CPS}.

As explained in Section~\ref{sec:prelim}, \AP{} holds if and only if the lattice width of $Q$ equals $2$, the smallest value possible for a lattice polytope with an interior lattice point.
Ewald's conjecture asks for $n$ such directions forming a lattice basis.

\subsection{Why the set $\E(P)$ matters}

\emph{Toric geometry.}
Let $D_v$ be the torus-invariant prime divisor of $X$ corresponding to the vertex $v$.
The lattice points of $P$ index the monomial sections of $\mathcal{O}_X(-K_X)$, and the divisor of the character $\chi^m$ is $\sum_{v}\pair{v}{m}D_v$ \cite[Proposition~4.1.2]{CLS}.
Hence $m\in\E(P)$ if and only if all coefficients of $\operatorname{div}(\chi^m)$ lie in $\{-1,0,1\}$: the condition \AP{} says that $X$ carries a non-constant character whose divisor of zeros and whose divisor of poles are both reduced.

\emph{Frobenius splittings.}
For an integer $q\ge2$, multiplication by $q$ on $N$ induces a toric endomorphism $F_q\colon X\to X$.
A splitting is an $\mathcal{O}_X$-linear retraction of the natural map $\mathcal{O}_X\to(F_q)_*\mathcal{O}_X$.
The variety $X$ is \emph{diagonally split at $q$} if the corresponding endomorphism of $X\times X$ admits a splitting compatible with the diagonal, meaning that it preserves the ideal sheaf of the diagonal \cite{Payne09}.
Payne calls
\[
\mathbb{F}_X=\{u\in M_\R:\ |\pair{v}{u}|\le1\ \text{for all }v\in\V\}=P\cap(-P)
\]
the \emph{diagonal splitting polytope}.
His criterion states that $X$ is diagonally split at $q$ if and only if $\operatorname{int}(\mathbb{F}_X)$ contains a representative of every class in $\frac1q M/M$ \cite[Theorem~1.2]{Payne09}.
For $q=2$, a point $u=m/2$ lies in this interior exactly when $|\pair{v}{m}|<2$ for all $v\in\V$.
Since these pairings are integers, this is equivalent to $m\in\E(P)$; under $\frac12 M/M\cong M/2M$, the criterion is therefore the surjectivity of the reduction map $\E(P)\to M/2M$.
For a complete toric variety, diagonal splitting implies that the section rings of nef line bundles are normally presented and Koszul \cite[Theorem~1.4]{Payne09}.
Chou, Hering, Payne, Tramel and Whitney \cite{CHPTW18} showed that a toric variety is diagonally split at every $q\ge2$ if and only if its primitive ray generators form a unimodular configuration.
For $X_Q$, this means that $Q$ is a unimodular smooth Fano polytope in the sense of \cite{Tu24}.
This link between $\E(P)$ and unimodularity motivates the present paper.

\emph{Symplectic geometry.}
McDuff \cite{McDuff11} introduced \emph{probes} to displace Lagrangian torus fibres of a symplectic toric manifold.
For a monotone polytope $P$ she proved that every point of the interior of $P$ other than the origin can be displaced by a probe if and only if $P$ satisfies the \emph{star Ewald condition}, a condition requiring suitable points of $\E(P)$ near every face of $P$ \cite[Theorem~1.2]{McDuff11}; in that case every toric fibre other than the central one is displaceable.
The star Ewald condition holds in dimension at most three \cite{McDuff11}, fails for some polytopes in every dimension $n\ge6$ \cite{McDuff11,CPS}, and holds for the duals of unimodular smooth Fano polytopes \cite{Tu24} and for deeply monotone polytopes \cite{CPS}.
Brendel \cite{Brendel} uses the property that every facet of $P$ meets $\E(P)\setminus\{0\}$ in his study of real Lagrangian tori; by \cite[Lemma~2.10]{CPS} this property follows from Ewald's conjecture.
All these conditions are statements about how many points $\E(P)$ has and where they lie, and \AP{} is the common minimal requirement behind them.

\subsection{Unimodular defect}

A finite set $S\subset N$ is called \emph{unimodular} if $\det(s_1,\dots,s_n)\in\{0,\pm1\}$ for all $s_1,\dots,s_n\in S$; the determinant is computed in a $\Z$-basis of $N$ and is well defined up to sign.
The following notion makes sense for every finite set of lattice vectors.

\begin{definition}[Unimodular defect]\label{def:ud}
Let $S\subset N$ be finite.
A \emph{transversal} of $S$ is a subset $B\subseteq S$ such that $S\setminus B$ is unimodular.
The \emph{unimodular defect} of $S$ is
\[
\ud(S)=\min\{|B|:\ B \text{ is a transversal of } S\}.
\]
For a smooth Fano polytope $Q$ we put $\ud(Q)=\ud(\V(Q))$ and say that $Q$ is a smooth Fano polytope \emph{of unimodular defect $\ud(Q)$}.
\end{definition}

Thus $\ud(S)$ is the least number of vectors that must be removed from $S$ to make it unimodular.
Every superset of a transversal is again a transversal; in particular, if $S$ is unimodular, every subset of $S$ is a transversal.
A transversal is called \emph{minimal} if it has the smallest possible number of elements, $\ud(S)$, and \emph{inclusion-minimal} if no proper subset is a transversal.
Smooth Fano polytopes of unimodular defect $0$ are the unimodular smooth Fano polytopes.
Every smooth Fano polytope has finite unimodular defect: the complement of a facet is a transversal, so
\begin{equation}\label{eq:defect-bound}
\ud(Q)\le|\V(Q)|-n\le 2n,
\end{equation}
where the second inequality is Casagrande's bound $|\V(Q)|\le3n$ \cite{Casagrande06}.
The classes $\{Q:\ \ud(Q)\le k\}$, $k=0,1,2,\dots$, therefore form an increasing filtration of the set of all smooth Fano polytopes, whose first member is the class for which Ewald's conjecture is known.
In view of \cite{CHPTW18}, $\ud(Q)$ measures how far $X_Q$ is from being diagonally split at all $q$.

Two vertices of $Q$ are \emph{adjacent} if they are joined by an edge of $Q$, and a set of vertices is \emph{independent} if its elements are pairwise non-adjacent.
The \emph{independent unimodular defect} $\udind(Q)$ is the least cardinality of an independent transversal of $\V(Q)$, and $\udind(Q)=\infty$ if no transversal is independent.
Clearly $\ud(Q)\le\udind(Q)$, and $\udind(Q)=\ud(Q)$ if $\ud(Q)\le1$.

\subsection{Main results}

A \emph{ridge} of $Q$ is a face of codimension two; since the facets of $Q$ are simplices, a ridge is spanned by $n-1$ vertices of a facet, and we identify it with this set of vertices.
For $W=\{w_1,\dots,w_{n-1}\}$, write $\det(v,W)=\det(v,w_1,\dots,w_{n-1})$ in a fixed lattice basis of $N$; its absolute value is independent of this basis and the ordering of $W$.
We call a ridge $W$ \emph{flat} if $|\det(v,W)|\le1$ for every vertex $v$ of $Q$.
A flat ridge immediately yields a nonzero element of $\E(P)$ (Proposition~\ref{prop:flat}); geometrically, it is an edge of $P$ whose primitive direction vector $e$ satisfies $\pm e\in P$ (Remark~\ref{rem:flat-meaning}).

\begin{thmx}\label{thm:A}
Let $Q$ be a smooth Fano polytope of dimension $n\ge2$ and unimodular defect $\ud(Q)\le 2$.
Then $Q$ has a flat ridge.
In particular, the smooth reflexive polytope $P=Q^{*}$ satisfies the antipodal condition \AP.
\end{thmx}

A transversal with two elements consists either of two non-adjacent or of two adjacent vertices, and Theorem~\ref{thm:A} covers both cases; the flat ridge is given explicitly in each case (Table~\ref{tab:certificates}).
The case of two adjacent vertices is the main point of the proof.
By~\eqref{eq:defect-bound}, Theorem~\ref{thm:A} applies in particular to all smooth Fano polytopes with at most $n+2$ vertices, that is, to the smooth toric Fano varieties of Picard number at most $2$, which were classified by Kleinschmidt \cite{Kleinschmidt88}.

Our second result concerns independent transversals.
It rests on the structure of independent sets of vertices (Proposition~\ref{prop:indep}), which is essentially due to Nill \cite{Nill05}: two non-adjacent vertices $x,u$ satisfy $u=-x$ or $x+u\in\V$, and three pairwise non-adjacent vertices $v_1,v_2,v_3$ satisfy $v_1+v_2+v_3=0$.

\begin{thmx}\label{thm:B}
Let $Q$ be a smooth Fano polytope with vertex set $\V$.
\begin{enumerate}[label=\textup{(\alph*)}]
\item\label{C:size} Every independent transversal of $\V$ has at most three elements.
If $Q$ is not unimodular, every independent transversal has at most two elements.
Equivalently, an independent transversal with three elements exists only if $Q$ is unimodular, and it is then not minimal, since $\emptyset$ is a transversal.
\item\label{C:defect} If $Q$ has an independent transversal, then $\udind(Q)=\ud(Q)\in\{0,1,2\}$.
\item\label{C:AP} If $Q$ has an independent transversal, then $Q$ has a flat ridge, and $P=Q^{*}$ satisfies \AP.
\item\label{C:forms} A minimal independent transversal is of one of the forms
\[
\emptyset,\qquad \{x\}\ \ (-x\notin\V),\qquad \{x,-x\},\qquad \{x,u\}\ \ (x+u\in\V,\ -x\notin\V,\ -u\notin\V).
\]
\end{enumerate}
\end{thmx}

Thus the independent unimodular defect takes only the values $0,1,2,\infty$, and restricting to independent transversals never leads beyond unimodular defect two.
Example~\ref{ex:adjacent} shows that $\udind(Q)=\infty$ occurs for polytopes of unimodular defect $2$; these are covered by Theorem~\ref{thm:A} but not by Theorem~\ref{thm:B}.
Example~\ref{ex:hexagon} shows that independent transversals with three elements do occur for unimodular polytopes.

All arguments use two elementary tools.
The first is the \emph{height} $h_G(v)=\pair{v}{n_G}$ of a vertex $v$ over a facet $G$, where $n_G$ is the lattice vector with $\pair{g}{n_G}=1$ for $g\in G$: heights of vertices are integers, at most $1$, and equal to $1$ only on $G$ (Section~\ref{sec:prelim}).
The second is that for a facet $G$ and $x\in G$, the $x$-coordinate of a vector $z$ with respect to the basis $G$ equals $\pm\det(z,G\setminus\{x\})$ (Lemma~\ref{lem:coord}).
A flat ridge $W=G\setminus\{x\}$ is thus a facet basis $G$ together with a basis vector $x$ such that every vertex of $Q$ has $x$-coordinate in $\{-1,0,1\}$; the coordinate function itself is the desired point of $\E(P)$.

\subsection{Organisation}
Section~\ref{sec:prelim} collects preliminaries and Section~\ref{sec:ud} basic properties of the unimodular defect.
Section~\ref{sec:flat} introduces flat ridges, Section~\ref{sec:main} proves Theorem~\ref{thm:A}, and Section~\ref{sec:indep} proves Theorem~\ref{thm:B}.
Section~\ref{sec:examples} contains examples.
Section~\ref{sec:data} reports the computations and contains further remarks and open questions.

\section{Preliminaries}\label{sec:prelim}

Throughout, $Q\subset N_\R$ is a smooth Fano polytope of dimension $n\ge 2$ with vertex set $\V$, and $P=Q^*$.
We identify facets and ridges with their vertex sets; they contain $n$ and $n-1$ vertices, respectively.

For $z_1,\dots,z_n\in N_\R$ we write $\det(z_1,\dots,z_n)$ for the determinant of their coordinates in a fixed $\Z$-basis of $N$; for a set $W=\{w_1,\dots,w_{n-1}\}$ we write $\det(z,W)=\det(z,w_1,\dots,w_{n-1})$.
Only absolute values and ratios of such determinants will matter, so the choice of the basis and the order of $W$ play no role.
For a facet $G$ we have $|\det G|=1$.

Let $G$ be a facet.
Since $G$ is a $\Z$-basis of $N$, there is a unique $n_G\in M$ with $\pair{g}{n_G}=1$ for all $g\in G$.
The hyperplane $\{\pair{\cdot}{n_G}=1\}$ supports $Q$ along the facet $\conv(G)$, and the origin lies on the side $\{\pair{\cdot}{n_G}<1\}$; hence $Q\subseteq\{\pair{\cdot}{n_G}\le1\}$.
We call $h_G(z)=\pair{z}{n_G}$ the \emph{height} of $z\in N_\R$ over $G$.
Heights are linear in $z$, and they are integers on $N$.

Thus, for every vertex $v$ of $Q$, $h_G(v)\le1$, with equality exactly when $v\in G$; in particular $h_G(v)\le0$ for $v\notin G$.

The cones $\cone(G)$ over the facets cover $N_\R$, because the origin is an interior point of $Q$.
If $z\in N\cap\cone(G)$, then
\begin{equation}\label{eq:cone-coeff}
z=\sum_{g\in G}c_g\,g\qquad\text{with } c_g\in\Z_{\ge0}\ \text{ and }\ \sum_{g\in G}c_g=h_G(z),
\end{equation}
since $G$ is a $\Z$-basis and $h_G(g)=1$ for $g\in G$.
If also $z\in Q$, then $0\le\sum_gc_g\le1$, so $z=0$ or $z$ is a vertex of $G$.
In particular, the only boundary lattice points of $Q$ are its vertices.

Let $G$ be a facet and $x\in G$.
We denote by $x^*_G\in M$ the element of the basis dual to $G$ that corresponds to $x$, so that $\pair{x}{x^*_G}=1$ and $\pair{w}{x^*_G}=0$ for $w\in G\setminus\{x\}$.
Thus $\pair{z}{x^*_G}$ is the \emph{$x$-coordinate} of $z$ with respect to the basis $G$.

\begin{lemma}\label{lem:coord}
Let $G$ be a facet, $x\in G$ and $W=G\setminus\{x\}$.
Then for all $z\in N_\R$
\[
\pair{z}{x^*_G}=\frac{\det(z,W)}{\det(x,W)},\qquad\text{and hence}\qquad |\pair{z}{x^*_G}|=|\det(z,W)|.
\]
\end{lemma}

\begin{proof}
Both sides of the first identity are linear in $z$ and agree on the basis $G$: they equal $1$ at $z=x$ and $0$ at $z=w\in W$.
The second identity follows from $|\det(x,W)|=|\det G|=1$.
\end{proof}

\begin{lemma}[\cite{Obro07}]\label{lem:wall}
Let $W$ be a ridge and let $G=W\cup\{x\}$ and $G'=W\cup\{x'\}$ be the two facets containing it.
Then
\[
x'=-x+\sum_{w\in W}a_w\,w\qquad\text{with } a_w\in\Z .
\]
In particular $x'^*_{G'}=-x^*_G$, and $h_{G'}(x)=h_G(x')=s-1\le 0$, where $s=\sum_{w\in W}a_w$.
\end{lemma}

\begin{proof}
Write $x'=c\,x+\sum_{w}a_w w$ in the basis $G$; the coefficients are integers.
By Lemma~\ref{lem:coord}, $|c|=|\det(x',W)|=|\det G'|=1$.
Suppose $c=1$, and let $p$ be the sum of the elements of $W$.
For small $\varepsilon>0$ the point $p+\varepsilon x'=\varepsilon x+\sum_w(1+\varepsilon a_w)w$ has only positive coordinates in the basis $G$ and only positive coordinates in the basis $G'$, so it lies in the interior of both $\cone(G)$ and $\cone(G')$.
This is impossible, since the cones over two distinct facets of $Q$ have disjoint interiors.
Hence $c=-1$.

The functional $-x^*_G$ vanishes on $W$ and takes the value $1$ at $x'$, so it equals $x'^*_{G'}$.
Applying $n_G$ to the relation gives $h_G(x')=s-1$, which is $\le0$ because $x'\notin G$; applying $n_{G'}$ to $x=-x'+\sum_w a_ww$ gives $h_{G'}(x)=s-1$.
\end{proof}

By Lemma~\ref{lem:wall} the functional $x^*_G$ depends, up to sign, only on the ridge $W=G\setminus\{x\}$: it is the primitive lattice vector, unique up to sign, that vanishes on $W$.
We denote it by
\begin{equation}\label{eq:mW}
m_W:=x^*_G\in M\qquad(\text{defined up to sign}).
\end{equation}

The \emph{lattice width} of $Q$ is
$\min_{m\in M\setminus\{0\}}\left(\max_{z\in Q}\pair{z}{m}-\min_{z\in Q}\pair{z}{m}\right).$
Since $0$ is an interior point of $Q$, every nonzero $m\in M$ takes both positive and negative values on its vertices.
These values are integers, so the width in the direction $m$ is at least $2$, with equality exactly when $m\in\E(P)$.
Consequently, $Q$ satisfies \AP{} if and only if its lattice width equals $2$.

\section{Unimodular defect}\label{sec:ud}

Recall from Definition~\ref{def:ud} the notions of a unimodular set, of a transversal and of the unimodular defect $\ud(S)$.
Once a transversal $B$ is fixed, we call its elements \emph{bad} and the elements of $S\setminus B$ \emph{good}.

\begin{remark}\label{rem:TU}
A set $S$ is unimodular if and only if every linearly independent $n$-element subset of $S$ is a $\Z$-basis of $N$.
If $S$ contains a $\Z$-basis $F$ of $N$, as the vertex set of a smooth Fano polytope does, this holds if and only if the matrix of the coordinates of the elements of $S\setminus F$ in the basis $F$ is totally unimodular \cite[\S19.1]{Schrijver}.
In matroid terms, such sets are the unimodular representations of regular matroids, whose structure is described by Seymour's decomposition theorem \cite{Seymour80}; see \cite{Tu24} for the consequences for unimodular smooth Fano polytopes.
Let $\mathcal{H}(S)$ be the hypergraph on $S$ whose edges are the $n$-element subsets $T\subseteq S$ with $|\det T|\ge2$.
A set $B\subseteq S$ is a transversal of $S$ if and only if it meets every edge of $\mathcal{H}(S)$; so $\ud(S)$ is the transversal number of $\mathcal{H}(S)$, which explains the terminology.
\end{remark}

\begin{example}
In $N=\Z^2$ let $S=\{e_1,\ e_2,\ e_1+e_2,\ e_1-e_2\}$.
The only pair with determinant $\pm2$ is $\{e_1+e_2,e_1-e_2\}$, so $\{e_1+e_2\}$ and $\{e_1-e_2\}$ are the minimal transversals and $\ud(S)=1$.
By contrast, all five smooth Fano polygons are unimodular.
\end{example}

\begin{lemma}\label{lem:basic}
Let $S\subset N$ be finite.
\begin{enumerate}[label=\textup{(\alph*)}]
\item\label{basic:mono} If $S'\subseteq S$, then $\ud(S')\le\ud(S)$.
\item\label{basic:GL} $\ud(\varphi(S))=\ud(S)$ for every automorphism $\varphi$ of the lattice $N$.
\item\label{basic:bound} If $S$ contains a $\Z$-basis of $N$, then $\ud(S)\le|S|-n$.
\item\label{basic:exchange} Let $B$ be a transversal of $S$ and $b\in B$ with $-b\in S\setminus B$. Then $B\setminus\{b\}$ is a transversal of $S$.
In particular, if $-b\in S\setminus B$ for every $b\in B$, then $S$ is unimodular.
\end{enumerate}
\end{lemma}

\begin{proof}
\ref{basic:mono} If $B$ is a transversal of $S$, then $B\cap S'$ is a transversal of $S'$, because $S'\setminus B\subseteq S\setminus B$.
\ref{basic:GL} is clear, since $|\det|$ is invariant under automorphisms of $N$.
\ref{basic:bound} If $F\subseteq S$ is a $\Z$-basis, then $F$ is unimodular, so $S\setminus F$ is a transversal.

\ref{basic:exchange} Let $T\subseteq S\setminus(B\setminus\{b\})$ be an $n$-element set; we must show $|\det T|\le1$.
If $b\notin T$, then $T\subseteq S\setminus B$ and we are done.
If $b\in T$ and $-b\in T$, then $\det T=0$.
If $b\in T$ and $-b\notin T$, then $T'=(T\setminus\{b\})\cup\{-b\}$ is an $n$-element subset of $S\setminus B$, and $|\det T|=|\det T'|\le1$.
\end{proof}

For a smooth Fano polytope, Lemma~\ref{lem:basic}\ref{basic:bound} applied to a facet gives~\eqref{eq:defect-bound}.
If $Q'\subseteq Q$ are smooth Fano polytopes with $\V(Q')\subseteq\V(Q)$, for instance if $X_{Q}\to X_{Q'}$ is a toric blow-up, then $\ud(Q')\le\ud(Q)$ by Lemma~\ref{lem:basic}\ref{basic:mono}.

\section{Flat ridges}\label{sec:flat}

\begin{definition}\label{def:flat}
A ridge $W$ of $Q$ is \emph{flat} if $|\det(v,W)|\le1$ for every vertex $v$ of $Q$.
\end{definition}

\begin{proposition}\label{prop:flat}
If $W$ is a flat ridge of $Q$, then the lattice vector $m_W$ of~\eqref{eq:mW} is a nonzero element of $\E(P)$.
In particular $P$ satisfies \AP.
\end{proposition}

\begin{proof}
Let $G=W\cup\{x\}$ be a facet, so that $m_W=x^*_G$.
By Lemma~\ref{lem:coord}, $W$ is flat if and only if every vertex has $x$-coordinate in $\{-1,0,1\}$ with respect to $G$.
Then $\pair{x}{m_W}=1$, so $m_W\ne0$, and $|\pair{v}{m_W}|=|\det(v,W)|\le1$ for all $v\in\V$ by Lemma~\ref{lem:coord}.
\end{proof}

\begin{remark}\label{rem:flat-meaning}
Let $G=W\cup\{x\}$ and $G'=W\cup\{x'\}$ be the two facets containing $W$.
The ridge $W$ is dual to the edge of $P$ joining the vertices $-n_G$ and $-n_{G'}$.
The difference $n_G-n_{G'}$ vanishes on $W$ and takes the value $1-h_{G'}(x)=2-s$ at $x$ (Lemma~\ref{lem:wall}), so $n_G-n_{G'}=(2-s)\,m_W$.
Hence $m_W$ is the primitive direction vector of this edge, and $2-s\ge1$ is its lattice length.
Thus \emph{$Q$ has a flat ridge if and only if $P$ has an edge whose primitive direction vector $e$ satisfies $\pm e\in P$.}

\end{remark}

If $Q$ is unimodular, every ridge is flat.
The next lemma extends this observation and is the starting point of all our arguments.

\begin{lemma}\label{lem:good-values}
Let $B$ be a transversal of $\V$ and let $W$ be a ridge with $W\cap B=\emptyset$.
Then $|\det(v,W)|\le1$ for every $v\in\V\setminus B$.
Consequently $W$ is flat as soon as $|\det(u,W)|\le1$ for every $u\in B$.
\end{lemma}

\begin{proof}
Let $v\in\V\setminus B$.
If $v\in W$, then $\det(v,W)=0$.
Otherwise $\{v\}\cup W$ is an $n$-element subset of the unimodular set $\V\setminus B$.
\end{proof}

\section{Unimodular defect at most two: proof of Theorem~\ref{thm:A}}\label{sec:main}

Let $B$ be a transversal of $\V$ with $|B|\le2$.
There are four cases: $B$ has at most one element; $B=\{x,-x\}$; $B$ consists of two non-adjacent vertices which are not antipodal; $B$ consists of two adjacent vertices.
Proposition~\ref{prop:one} treats the first two cases, Proposition~\ref{prop:nonadjacent} the third, and Proposition~\ref{prop:adjacent} the fourth.
The resulting flat ridges are listed in Table~\ref{tab:certificates}.

\subsection{At most one bad vertex up to sign}

\begin{proposition}\label{prop:one}
Let $B$ be a transversal of $\V$.
\begin{enumerate}[label=\textup{(\alph*)}]
\item If $B=\emptyset$, then every ridge of $Q$ is flat.
\item If $B=\{x\}$ or $B=\{x,-x\}$, then $G\setminus\{x\}$ is a flat ridge for every facet $G$ containing $x$.
\end{enumerate}
\end{proposition}

\begin{proof}
(a) is Lemma~\ref{lem:good-values} with $B=\emptyset$.
(b) Let $G$ be a facet containing $x$ and $W=G\setminus\{x\}$.
The vertex $-x$ is not in $G$, because $G$ is linearly independent; hence $W\cap B=\emptyset$.
Moreover $|\det(x,W)|=|\det G|=1$ and $|\det(-x,W)|=1$.
By Lemma~\ref{lem:good-values}, $W$ is flat.
\end{proof}

\subsection{Two non-adjacent bad vertices}

\begin{lemma}\label{lem:pair}
Let $x$ and $u$ be non-adjacent vertices of $Q$.
Then either $u=-x$, or $y:=x+u$ is a vertex of $Q$ which is adjacent to both $x$ and $u$.
\end{lemma}

This is the smooth case of \cite[Proposition~3.1]{Nill05}; see also the discussion before Proposition~\ref{prop:indep}.
We include the short proof in the smooth case.

\begin{proof}
Choose a facet $G$ with $x+u\in\cone(G)$ and write $x+u=\sum_{g\in G}c_g\,g$ as in~\eqref{eq:cone-coeff}, so that $c_g\in\Z_{\ge0}$ and $\sum_{g}c_g=h_G(x)+h_G(u)$.
Since $x$ and $u$ are not adjacent, they do not both belong to $G$.
Thus, one of the two heights is at most $0$ and the other at most $1$, so $\sum_g c_g\le1$.
Therefore either all $c_g$ vanish and $x+u=0$, or exactly one $c_g$ equals $1$ and $x+u=g$ is a vertex.

Assume $u\neq-x$, so that $y=x+u$ is a vertex; clearly $y\ne x$ and $y\ne u$.
Let $G'$ be any facet containing $y$.
Then $h_{G'}(x)+h_{G'}(u)=h_{G'}(y)=1$, and both summands are integers which are at most $1$; hence one of them equals $1$.
This means: \emph{every facet containing $y$ contains $x$ or $u$.}
If no facet contained both $y$ and $x$, every facet containing $y$ would contain $u$, contradicting the fact that the intersection of all facets containing the vertex $y$ is $\{y\}$ \cite[Theorem~2.7]{Ziegler}.
So $y$ is adjacent to $x$, and by symmetry to $u$.
\end{proof}

\begin{proposition}\label{prop:nonadjacent}
Let $B=\{x,u\}$ be a transversal of $\V$ such that $x$ and $u$ are not adjacent and $u\neq-x$.
Then $y=x+u$ is a good vertex adjacent to both $x$ and $u$, and $G\setminus\{x\}$ is a flat ridge for every facet $G$ containing $x$ and $y$.
\end{proposition}

\begin{proof}
By Lemma~\ref{lem:pair}, $y$ is a vertex adjacent to $x$; it is good because $y\ne x$ and $y\neq u$.
Let $G$ be a facet containing $x$ and $y$, and put $W=G\setminus\{x\}$.
Since $u$ is not adjacent to $x$, we have $u\notin G$, so $W\cap B=\emptyset$.
We have $|\det(x,W)|=1$, and since $y\in W$,
\[
\det(u,W)=\det(y,W)-\det(x,W)=-\det(x,W),
\]
so $|\det(u,W)|=1$.
Thus, by Lemma~\ref{lem:good-values}, $W$ is flat.
\end{proof}

\subsection{Two adjacent bad vertices}

\begin{proposition}\label{prop:adjacent}
Let $B=\{u_0,u_1\}$ be a transversal of $\V$ consisting of two adjacent vertices.
Let $G$ be any facet containing $u_0$ and $u_1$, and put $R=G\setminus\{u_0,u_1\}$.
Let
\[
G_1=R\cup\{u_1,y_1\}\qquad\text{and}\qquad G_0=R\cup\{u_0,y_0\}
\]
be the facets neighbouring $G$ across the ridges $R\cup\{u_1\}$ and $R\cup\{u_0\}$, respectively.
Then $y_0$ and $y_1$ are good, and at least one of the ridges $R\cup\{y_1\}$ and $R\cup\{y_0\}$ is flat.
\end{proposition}

\begin{proof}
The vertices $y_0,y_1$ do not belong to $G$, whereas $B\subseteq G$; hence $y_0$ and $y_1$ are good.
All elements of $R$ are good as well.

\emph{Step 1: coordinates of $y_1$ and $y_0$.}
The facet $G_1$ is obtained from $G$ by replacing $u_0$ with $y_1$, and $G_0$ by replacing $u_1$ with $y_0$.
Since $R$ is part of the lattice basis $G$, its integer span satisfies $\Z R=\spn_\R(R)\cap N$.
By Lemma~\ref{lem:wall} there are integers $a,b$ and vectors $r_0,r_1\in\spn_\R(R)\cap N$ such that
\begin{equation}\label{eq:y1y0}
y_1=-u_0+b\,u_1+r_1,\qquad y_0=a\,u_0-u_1+r_0 .
\end{equation}

\emph{Step 2: $|a|\le1$ or $|b|\le1$.}
If $y_1=y_0$, comparing the coefficients of $u_0$ and $u_1$ in~\eqref{eq:y1y0} gives $a=-1$ and $b=-1$.
Assume $y_1\neq y_0$.
Then $\{y_1,y_0\}\cup R$ is an $n$-element set of good vertices, so its determinant lies in $\{0,\pm1\}$.
On the other hand, by~\eqref{eq:y1y0} and multilinearity,
\[
\det(y_1,y_0,R)=\det\begin{pmatrix}-1&b\\ a&-1\end{pmatrix}\cdot\det(u_0,u_1,R)=\pm(1-ab).
\]
Hence $1-ab\in\{-1,0,1\}$, that is, $ab\in\{0,1,2\}$.
For integers $a,b$ this forces $|a|\le1$ or $|b|\le1$.

\emph{Step 3: the flat ridge.}
Suppose $|b|\le1$ and put $W=R\cup\{y_1\}$.
This is a ridge, since it is contained in the facet $G_1$, and all its elements are good.
By~\eqref{eq:y1y0},
\[
\det(u_0,R,y_1)=b\,\det(u_0,R,u_1)=\pm b,\qquad \det(u_1,R,y_1)=-\det(u_1,R,u_0)=\pm1 .
\]
So $|\det(u_0,W)|\le1$ and $|\det(u_1,W)|\le1$, and $W$ is flat by Lemma~\ref{lem:good-values}.
If $|a|\le1$, the same argument with $W=R\cup\{y_0\}\subseteq G_0$ gives $|\det(u_1,W)|=|a|\le1$ and $|\det(u_0,W)|=1$.
\end{proof}

\begin{figure}[ht]
\centering
\begin{tikzpicture}[scale=1.25]
  \fill[black!10] (-1,-1) rectangle (1,1);
  \fill[blue!10] (0,0) -- (3.3,0) -- (3.3,2.4) -- (0,2.4) -- cycle;
  \foreach \x in {-3,-2,-1,0,1,2,3}
    \foreach \y in {-2,-1,0,1,2}
      \fill[black!45] (\x,\y) circle (1pt);
  \draw[dashed,black!55] (-1,-2.4) -- (-1,2.4);
  \draw[dashed,black!55] (-3.4,-1) -- (3.3,-1);
  \draw[->,thick] (0,0) -- (1,0);
  \draw[->,thick] (0,0) -- (0,1);
  \draw[->,thick,red!70!black] (0,0) -- (-1,1);
  \draw[->,thick,red!70!black] (0,0) -- (2,-1);
  \node[anchor=south west] at (1.02,0.02) {$\bar u_0=(1,0)$};
  \node[anchor=south west] at (0.02,1.02) {$\bar u_1=(0,1)$};
  \node[anchor=south east,red!70!black] at (-1.05,1.02) {$\bar y_1=(-1,b)$};
  \node[anchor=north west,red!70!black] at (2.02,-1.05) {$\bar y_0=(a,-1)$};
  \node[black!70] at (2.5,1.5) {$\cone(\bar u_0,\bar u_1)$};
  \node[black!65,anchor=north] at (-1,-2.45) {$\alpha=-1$};
  \node[black!65,anchor=east] at (-3.45,-1) {$\beta=-1$};
  \node[black!65] at (-0.5,-0.55) {$[-1,1]^2$};
\end{tikzpicture}
\caption{The situation of Proposition~\ref{prop:adjacent} after projecting along $R$, drawn for $b=1$ and $a=2$. The ray $\bar y_1$ lies on the line $\alpha=-1$ and the ray $\bar y_0$ on the line $\beta=-1$. Since $1-ab\in\{-1,0,1\}$, at least one of them lies in the square $[-1,1]^2$; here it is $\bar y_1$, and $R\cup\{y_1\}$ is a flat ridge.}
\label{fig:planar}
\end{figure}
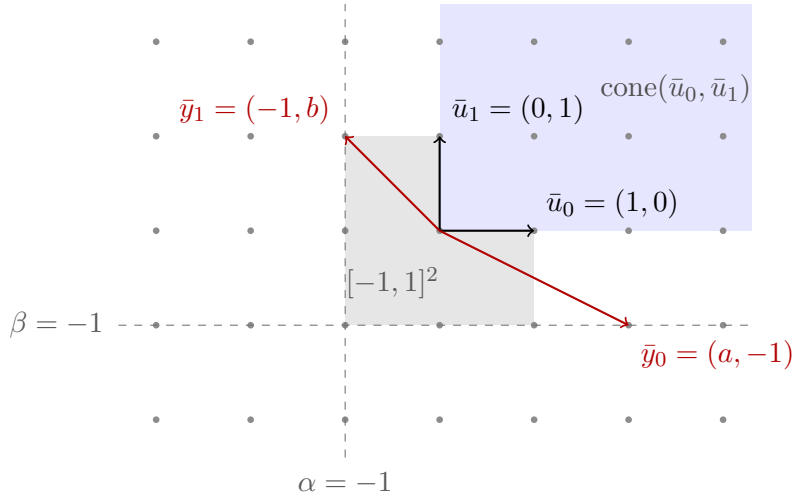

Projecting along $\Z R$ gives $N/\Z R\cong\Z^2$, with $\bar u_0=(1,0)$, $\bar u_1=(0,1)$, $\bar y_1=(-1,b)$ and $\bar y_0=(a,-1)$ (see Figure~\ref{fig:planar}).

\begin{table}[ht]
\centering
\renewcommand{\arraystretch}{1.25}
\begin{tabular}{@{}lll@{}}
\toprule
transversal $B$ & facet & flat ridge $W$\\
\midrule
$\emptyset$ & any $G$ & any ridge\\
$\{x\}$ & any $G\ni x$ & $G\setminus\{x\}$\\
$\{x,-x\}$ & any $G\ni x$ & $G\setminus\{x\}$\\
$\{x,u\}$, non-adjacent, $u\ne-x$ & any $G\supseteq\{x,\,x+u\}$ & $G\setminus\{x\}$\\
$\{u_0,u_1\}$, adjacent, $|b|\le1$ & $G_1=R\cup\{u_1,y_1\}$ & $R\cup\{y_1\}$\\
$\{u_0,u_1\}$, adjacent, $|a|\le1$ & $G_0=R\cup\{u_0,y_0\}$ & $R\cup\{y_0\}$\\
\bottomrule
\end{tabular}
\medskip
\caption{The flat ridges of Theorem~\ref{thm:A}. In the last two rows, $G=R\cup\{u_0,u_1\}$ is any facet containing $B$, and $y_1=-u_0+b\,u_1+r_1$, $y_0=a\,u_0-u_1+r_0$ are the new vertices of the two neighbouring facets, as in Proposition~\ref{prop:adjacent}.}
\label{tab:certificates}
\end{table}

\subsection{Proof of Theorem~\ref{thm:A}}

\begin{proof}[Proof of Theorem~\ref{thm:A}]
Let $B$ be a transversal of $\V$ with $|B|\le2$.
If $B=\emptyset$, $B=\{x\}$ or $B=\{x,-x\}$, Proposition~\ref{prop:one} gives a flat ridge.
If $B=\{x,u\}$ with $x,u$ non-adjacent and $u\neq-x$, Proposition~\ref{prop:nonadjacent} gives a flat ridge.
If $B$ consists of two adjacent vertices, Proposition~\ref{prop:adjacent} gives a flat ridge.
By Proposition~\ref{prop:flat}, $P$ satisfies \AP.
\end{proof}

\section{Independent sets of vertices and independent transversals}\label{sec:indep}

\subsection{Independent sets of vertices}

We recall the following consequences of Nill's results on simplicial reflexive polytopes \cite{Nill05}.
Part~\ref{B:pair} is the smooth case of \cite[Proposition~3.1]{Nill05}.
Parts~\ref{B:anti} and~\ref{B:hex} follow from \cite[Lemma~5.11]{Nill05}, using the fact that the only boundary lattice points of a smooth Fano polytope are its vertices; part~\ref{B:alpha} follows from part~\ref{B:hex}.
The numbering refers to the arXiv version \texttt{math/0405448}.
We include elementary proofs in the smooth case.

\begin{proposition}[Nill]\label{prop:indep}
Let $Q$ be a smooth Fano polytope.
\begin{enumerate}[label=\textup{(\alph*)}]
\item\label{B:pair} If $x$ and $u$ are non-adjacent vertices, then either $u=-x$, or $x+u$ is a vertex of $Q$ which is adjacent to both $x$ and $u$.
\item\label{B:anti} If $u$ and $-u$ are both vertices, then every other vertex is adjacent to $u$ or to $-u$.
\item\label{B:hex} If $v_1,v_2,v_3$ are pairwise non-adjacent vertices, then $v_1+v_2+v_3=0$, and $\pm v_1,\pm v_2,\pm v_3$ are six distinct vertices of $Q$.
\item\label{B:alpha} $Q$ has at most three pairwise non-adjacent vertices.
\end{enumerate}
\end{proposition}

\begin{proof}[Proof of Proposition~\ref{prop:indep}]
Part~\ref{B:pair} is Lemma~\ref{lem:pair}.

\ref{B:anti} Let $x\neq\pm u$ be a vertex, and suppose that $x$ is adjacent to neither $u$ nor $-u$.
By Lemma~\ref{lem:pair} applied to $\{x,u\}$ and $\{x,-u\}$, the vectors $y=x+u$ and $y'=x-u$ are vertices.
Let $G$ be a facet containing $x$.
Then $h_G(y)+h_G(y')=2h_G(x)=2$, so $y,y'\in G$.
The three vertices $x,y,y'$ of $G$ are distinct (because $u\neq0$), and $y+y'-2x=0$.
This contradicts the linear independence of $G$.

\ref{B:hex} By~\ref{B:anti}, no two of the $v_i$ are antipodal.
By Lemma~\ref{lem:pair}, $y_{12}=v_1+v_2$, $y_{13}=v_1+v_3$ and $y_{23}=v_2+v_3$ are vertices, and $y_{ij}$ is adjacent to $v_i$ and to $v_j$.
For a facet $G$ we abbreviate $h_i=h_G(v_i)$; then $h_G(y_{ij})=h_i+h_j$, and all heights of vertices are at most $1$.

\emph{Step 1: $y_{12}$ is different from $v_3$ and not adjacent to $v_3$.}
First, $y_{12}\ne v_3$, because $y_{12}$ is adjacent to $v_1$ and $v_3$ is not.
Suppose a facet $G$ contains $y_{12}$ and $v_3$.
Then $h_3=1$ and $h_1+h_2=1$.
But $h_1+h_3=h_G(y_{13})\le1$ gives $h_1\le0$, and likewise $h_2\le0$; so $h_1+h_2\le0$, a contradiction.
By symmetry, $y_{23}$ is different from $v_1$ and not adjacent to $v_1$.

\emph{Step 2: $t:=v_1+v_2+v_3$ is zero.}
Since $t=y_{12}+v_3$, Step~1 and Lemma~\ref{lem:pair} show that either $t=0$ or $t$ is a vertex.
Suppose $t$ is a vertex.
Since also $t=y_{23}+v_1$ with $y_{23}$ and $v_1$ non-adjacent, Lemma~\ref{lem:pair} shows that $t$ is adjacent to $v_1$.
Let $G$ be a facet containing $t$ and $v_1$.
Then $h_1=1$ and $h_1+h_2+h_3=h_G(t)=1$, so $h_2+h_3=0$.
Moreover $1+h_2=h_G(y_{12})\le1$ and $1+h_3=h_G(y_{13})\le1$, so $h_2\le0$ and $h_3\le0$.
Hence $h_2=h_3=0$, so $h_G(y_{12})=h_G(y_{13})=1$, that is, $y_{12},y_{13}\in G$.
Now $v_1$, $t$, $y_{12}$, $y_{13}$ are elements of $G$ satisfying $y_{12}+y_{13}=2v_1+v_2+v_3=v_1+t$.
They are pairwise distinct: $y_{12}\ne y_{13}$ since $v_2\neq v_3$; $y_{1j}\neq v_1$ since $v_j\ne0$; $y_{12}\neq t$ and $y_{13}\neq t$ since $v_3\ne0$ and $v_2\neq0$; and $t\ne v_1$ since $v_2\neq-v_3$.
This is a non-trivial linear relation among elements of the basis $G$, a contradiction.
Hence $t=0$.

\emph{Step 3.}
Since $t=0$, the vectors $-v_3=y_{12}$, $-v_2=y_{13}$ and $-v_1=y_{23}$ are vertices, and the six vertices $\pm v_i$ are distinct because the $v_i$ are distinct and no two of them are antipodal.

\ref{B:alpha} If $v_1,v_2,v_3,v_4$ were four pairwise non-adjacent vertices, then~\ref{B:hex} applied to $\{v_1,v_2,v_3\}$ and to $\{v_1,v_2,v_4\}$ would give $v_3=-v_1-v_2=v_4$.
\end{proof}

\subsection{Independent transversals: proof of Theorem~\ref{thm:B}}

\begin{proof}[Proof of Theorem~\ref{thm:B}]
\ref{C:size}
By Proposition~\ref{prop:indep}\ref{B:alpha}, every independent transversal $B$ has $|B|\le3$.
Suppose $|B|=3$, say $B=\{v_1,v_2,v_3\}$.
By Proposition~\ref{prop:indep}\ref{B:hex} the vectors $-v_1,-v_2,-v_3$ are vertices of $Q$, and none of them belongs to $B$.
So $-b\in\V\setminus B$ for every $b\in B$, and $\V$ is unimodular by Lemma~\ref{lem:basic}\ref{basic:exchange}.
This proves all statements of~\ref{C:size}.

\ref{C:defect}
By~\ref{C:size}, either $Q$ has an independent transversal with at most two elements, or $Q$ is unimodular and $\emptyset$ is an independent transversal.
Hence $\ud(Q)\le\udind(Q)\le2$.
If $\ud(Q)\le1$, a minimal transversal has at most one element and is independent, so $\udind(Q)=\ud(Q)$.
If $\ud(Q)=2$, then $2=\ud(Q)\le\udind(Q)\le2$.

\ref{C:AP}
A minimal independent transversal has one of the forms $\emptyset$, $\{x\}$, $\{x,-x\}$, or $\{x,u\}$ with $x,u$ non-adjacent and $u\ne-x$; so Propositions~\ref{prop:one} and~\ref{prop:nonadjacent} give a flat ridge, and Proposition~\ref{prop:flat} gives \AP.
(Proposition~\ref{prop:adjacent} is not needed here.)

\ref{C:forms}
Let $B$ be a minimal independent transversal; by~\ref{C:defect}, $|B|\le2$.
If $B=\{x\}$ and $-x\in\V$, then $-x\in\V\setminus B$, and the exchange property (Lemma~\ref{lem:basic}\ref{basic:exchange}) shows that $\emptyset$ is a transversal, contradicting minimality.
Let $B=\{x,u\}$ with $u\neq-x$.
Then $x+u\in\V$ by Lemma~\ref{lem:pair}.
If $-x\in\V$, then $-x\notin B$ because $-x\ne u$, and the exchange property shows that $\{u\}$ is a transversal, which is independent and smaller than $B$; this is a contradiction.
The same argument shows $-u\notin\V$.
\end{proof}

\section{Examples}\label{sec:examples}

All claims in the following examples were verified by exact integer computations; the small ones are easily checked by hand.

\begin{example}[Unimodular defect one]\label{ex:defect1}
For $n\ge3$ and $2\le a\le n-1$ let
\[
\V=\{e_1,\dots,e_{n-1},\ \ x=-(e_1+\dots+e_{n-1})+a\,e_n,\ \ e_n,\ \ -e_n\}\subset\Z^n .
\]
This is the vertex set of a smooth Fano polytope $Q$ with $n+2$ vertices; the corresponding toric variety is the projective bundle $\PP_{\PP^{n-1}}(\mathcal{O}\oplus\mathcal{O}(a))$ \cite{Kleinschmidt88}.
We have $\det(e_1,\dots,e_{n-1},x)=a\ge2$, so $Q$ is not unimodular, while $\V\setminus\{x\}=\{e_1,\dots,e_n,-e_n\}$ is unimodular.
Hence $\ud(Q)=\udind(Q)=1$.
The facet $G=\{e_2,\dots,e_n,x\}$ contains $x$, and $G\setminus\{x\}$ is a flat ridge by Proposition~\ref{prop:one}.
\end{example}

\begin{example}[Unimodular defect two, non-adjacent bad vertices]\label{ex:nonadjacent}
In $\Z^6$ let
\[
a=(-1,0,-1,0,1,1),\qquad b=(0,-1,0,-1,1,1),\qquad x=(0,0,0,0,-1,1),
\]
and $\V=\{e_1,\dots,e_6,\,-e_6,\,a,\,b,\,x,\,-x\}$.
Then $\V$ is the vertex set of a smooth Fano polytope $Q$ with $11$ vertices and $45$ facets.
Its dual $P=Q^*$ is unimodularly equivalent to the smooth reflexive polytope with ID \texttt{F.6D.2386} in the Lorenz--Paffenholz list \cite{LP}.
Exactly $12$ six-element subsets of $\V$ have $|\det|\ge2$ (all with $|\det|=2$): the sets $T\cup\{x\}$ and $T\cup\{-x\}$, where $T$ runs over the five-element subsets of $\{e_1,e_2,e_3,e_4,a,b\}$.
Hence no single vertex is a transversal, $\{x,-x\}$ is an independent transversal, and $\ud(Q)=\udind(Q)=2$.
The facet $G=\{e_1,e_2,e_3,e_4,e_6,x\}$ contains $x$, and the $x$-coordinate with respect to $G$ is $m=-e_5^*$.
Its values are $1$ on $x$, $-1$ on $-x$, $e_5$, $a$, $b$, and $0$ on the other vertices, so $G\setminus\{x\}$ is a flat ridge (Proposition~\ref{prop:one}).
\end{example}

\begin{example}[Unimodular defect two, only adjacent bad vertices]\label{ex:adjacent}
In $\Z^5$ let
\[
p=(-1,0,0,0,1),\qquad q=(-1,-1,0,0,-1),\qquad r=(-2,-2,-1,-1,-2),
\]
and $\V=\{e_1,\dots,e_5,\,-e_1,\,p,\,q,\,r\}$.
Then $\V$ is the vertex set of a smooth Fano polytope $Q$ with $9$ vertices and $24$ facets.
Its dual $P=Q^*$ is unimodularly equivalent to the smooth reflexive polytope with ID \texttt{F.5D.0052} in the Lorenz--Paffenholz list \cite{LP}.
Exactly $12$ five-element subsets of $\V$ have $|\det|\ge2$, namely
\begin{gather*}
\{e_1,e_2,e_3,e_4,r\},\ \{e_1,e_3,e_4,e_5,r\},\ \{e_1,e_3,e_4,p,r\},\ \{e_2,e_3,e_4,e_5,r\},\\
\{e_2,e_3,e_4,-e_1,r\},\ \{e_2,e_3,e_4,p,q\},\ \{e_2,e_3,e_4,p,r\},\ \{e_2,e_3,p,q,r\},\\
\{e_2,e_4,p,q,r\},\ \{e_3,e_4,e_5,-e_1,r\},\ \{e_3,e_4,e_5,p,r\},\ \{e_3,e_4,-e_1,p,r\},
\end{gather*}
with $|\det|=4$ for $\{e_2,e_3,e_4,p,r\}$ and $|\det|=2$ for the others.
No vertex lies in all twelve sets, and $\{e_3,e_4\}$ meets all of them, so $\ud(Q)=2$.
The only pairs of non-adjacent vertices are $\{e_1,-e_1\}$, $\{e_1,p\}$ and $\{e_5,-e_1\}$ (note $e_1+p=e_5$ and $e_5-e_1=p$, in accordance with Lemma~\ref{lem:pair}), and none of them is a transversal: the first two miss $\{e_2,e_3,e_4,e_5,r\}$ and the third misses $\{e_1,e_2,e_3,e_4,r\}$.
Therefore every transversal with two elements consists of adjacent vertices, and $\udind(Q)=\infty$.
Theorem~\ref{thm:B} does not apply to $Q$, but Theorem~\ref{thm:A} does.

Let us follow the proof of Proposition~\ref{prop:adjacent} for the transversal $B=\{u_0,u_1\}=\{e_2,e_3\}$ and the facet $G=\{e_1,\dots,e_5\}$, so that $R=\{e_1,e_4,e_5\}$.
The neighbouring facets are $G_1=\{e_1,e_3,e_4,e_5,q\}$ and $G_0=\{e_1,e_2,e_4,e_5,r\}$, and
\[
y_1=q=-e_2+0\cdot e_3-e_1-e_5,\qquad y_0=r=-2e_2-e_3-2e_1-e_4-2e_5,
\]
so $b=0$ and $a=-2$; indeed $1-ab=1$.
Since $|b|\le1$, the ridge $W=R\cup\{q\}=\{e_1,e_4,e_5,q\}$ is flat.
The corresponding point of $\E(P)$ is the $e_3$-coordinate with respect to $G_1$, which is $m_W=e_3^*$: its values are $1$ on $e_3$, $-1$ on $r$ and $0$ on all other vertices.
The other ridge $R\cup\{r\}$ is not flat, since $|\det(e_3,R,r)|=|a|=2$.
\end{example}

\begin{example}[The hexagon]\label{ex:hexagon}
Let $Q\subset\R^2$ be the hexagon with vertices $\pm e_1$, $\pm e_2$, $\pm(e_1+e_2)$, the smooth Fano polygon of the del Pezzo surface of degree $6$.
It is unimodular, so $\ud(Q)=\udind(Q)=0$.
The triple $\{e_1,\,e_2,\,-e_1-e_2\}$ is independent and, like every subset of $\V$, a transversal; it is an independent transversal with three elements which is not minimal, as in Theorem~\ref{thm:B}\ref{C:size}.
By Proposition~\ref{prop:indep}\ref{B:hex}, every independent triple in a smooth Fano polytope looks like this: together with its negative it is the image of the vertex set of the hexagon under an injective linear map $\Z^2\to N$.
\end{example}

\section{Computations and further questions}\label{sec:data}\label{sec:further}

\subsection{Computational results}

We computed the unimodular defect of all smooth Fano polytopes of dimension at most $6$, using the classification lists produced with {\O}bro's algorithm \cite{Obro07} and made available by Lorenz and Paffenholz \cite{LP}.
For each polytope we determined all $n$-element subsets of $\V$ of determinant at least $2$ in absolute value and then the transversals with at most three elements.
All determinants were computed exactly in integer arithmetic, and every polytope in the lists was checked again to be a smooth Fano polytope.
The results are collected in Table~\ref{tab:defects}.

\begin{table}[ht]
\centering
\renewcommand{\arraystretch}{1.15}
\begin{tabular}{@{}rrrrrr@{}}
\toprule
$n$ & polytopes & $\ud=0$ & $\ud=1$ & $\ud=2$ & $\udind=\infty$\\
\midrule
2 & 5 & 5 & 0 & 0 & 0\\
3 & 18 & 16 & 2 & 0 & 0\\
4 & 124 & 96 & 28 & 0 & 0\\
5 & 866 & 554 & 302 & 10 & 10\\
6 & 7\,622 & 4\,097 & 3\,269 & 256 & 182\\
\bottomrule
\end{tabular}
\medskip
\caption{Unimodular defect of the smooth Fano polytopes of dimension $n\le6$. No polytope has $\ud\ge3$. The last column counts the polytopes without an independent transversal.}
\label{tab:defects}
\end{table}

For $n\le9$, condition \AP{} was already known by computer: {\O}bro (for $n\le8$) and Paffenholz (for $n=9$) verified that every facet of every smooth reflexive polytope of dimension $n\le9$ contains a nonzero point of $\E(P)$ \cite[\S3.4]{Brendel}.
The strong form of Ewald's conjecture was verified by {\O}bro for $n\le7$ \cite{ObroThesis}, \cite[Theorem~3.9]{CPS}.
Table~\ref{tab:defects} therefore shows that for $n\le6$ the known antipodal condition has a uniform explanation by Theorem~\ref{thm:A}.

We record the following observations.
\begin{itemize}[leftmargin=2em]
\item Polytopes without an independent transversal first appear in dimension $5$: for all ten polytopes of unimodular defect $2$ in dimension $5$, every transversal with two elements consists of adjacent vertices, as in Example~\ref{ex:adjacent}.
In dimension $6$ this happens for $182$ of the $256$ polytopes of unimodular defect $2$.
The other $74$ have an independent transversal with two elements.
\item For $n\le 6$ we verified Proposition~\ref{prop:indep} directly on all polytopes ($41\,327$ pairs of non-adjacent vertices and $1\,274$ independent triples), and we verified the flat ridges of Table~\ref{tab:certificates} for every transversal with at most two elements and every admissible facet ($2\,987\,254$ cases for the first four rows of the table and $5\,577\,049$ for the last two).
\end{itemize}

\subsection{A criterion for larger transversals}
The arguments for non-adjacent bad vertices in Section~\ref{sec:main} are special cases of a criterion that applies to transversals of any size. For a facet $G$ and $x\in G$ let the \emph{tube} $T(G,x):=\Big\{c\,x+\sum_{w\in G\setminus\{x\}}\lambda_w w:\ c\in\R,\ \lambda_w\ge0\Big\}$ be the set of points whose $G$-coordinates are all non-negative except possibly the $x$-coordinate. The following lemma can be deduced from Øbro [\cite{Obro07}, Lemma 2.3]

\begin{lemma}\label{lem:F}
Let $G$ be a facet of $Q$, $x\in G$ and $W=G\setminus\{x\}$, and let $z\ne x$ be a primitive lattice point of $Q$ in $T(G,x)$.
Then the $x$-coordinate of $z$ with respect to $G$ is $0$ or $-1$.
\end{lemma}

\begin{proof}
Write $z=c\,x+\sum_{w\in W}\lambda_w w$ with integers $c$ and $\lambda_w\ge0$.
As $z\in Q$, we have $c+\sum_w\lambda_w=h_G(z)\le1$, so $c\le1$, and $c=1$ would force $z=x$.
Suppose $c\le-2$.
Let $G'=W\cup\{x'\}$ be the other facet containing $W$, with $x'=-x+\sum_wa_ww$ and $s=\sum_wa_w\le1$ as in Lemma~\ref{lem:wall}.
Then
\[
1\ \ge\ h_{G'}(z)=c\,(s-1)+\sum_w\lambda_w=|c|\,(1-s)+\sum_w\lambda_w\ \ge\ 2\,(1-s)+\sum_w\lambda_w ,
\]
so $s=1$ and $\sum_w\lambda_w\le1$.
If $\sum_w\lambda_w=0$, then $z=c\,x$ is not primitive.
If $\sum_w\lambda_w=1$, then $h_{G'}(z)=1$, so $z$ lies in the facet $\conv(G')$: $z=\alpha x'+\sum_w\beta_ww$ with $\alpha,\beta_w\ge0$ and $\alpha+\sum_w\beta_w=1$; substituting $x'=-x+\sum_wa_ww$ shows $c=-\alpha\in[-1,0]$, a contradiction.
\end{proof}

\begin{remark}\label{rem:tube}
Lemmas~\ref{lem:F}, \ref{lem:good-values} and~\ref{lem:coord} give a so-called \emph{Tube criterion}: \emph{let $B$ be a transversal, $G$ a facet and $x\in G$ such that $G\setminus\{x\}$ contains no bad vertex and every $u\in B\setminus\{x\}$ satisfies $u\in T(G,x)$ or $-u\in Q\cap T(G,x)$; then $G\setminus\{x\}$ is a flat ridge.}
Propositions~\ref{prop:one} and~\ref{prop:nonadjacent} are the cases $B\setminus\{x\}\subseteq\{-x\}$ and $B=\{x,u\}$ with $u=-x+(x+u)\in T(G,x)$.
A bad vertex $u\ne\pm x$ adjacent to $x$ can lie in $T(G,x)$ only if $u\in G$, so the criterion does not apply to adjacent bad vertices.
The hypothesis $\lambda_w\ge0$ in Lemma~\ref{lem:F} cannot be weakened to $\sum_w\lambda_w\ge0$: for the smooth Fano polytope with vertices $(-1,1,2)$, $(-1,0,0)$, $(-1,0,1)$, $(0,-1,0)$, $(0,0,-1)$, $(1,0,-1)$, the facet $G=\{g_1,g_2,x\}$ with $g_1=(-1,1,2)$, $g_2=(-1,0,0)$, $x=(0,0,-1)$ and the vertex $z=(0,-1,0)=-2x-g_1+g_2$ give $\sum_w\lambda_w=0$ but $c=-2$.
\end{remark}

\subsection{The Ewald condition}
A flat ridge supplies one element of a lattice basis contained in $\E(P)$.
For independent transversals, the remaining elements can be looked for in a hyperplane.

\begin{proposition}\label{prop:split}
Let $B$ be an independent transversal with one or two elements, let $x\in B$, and put
\[
K(B,x)=\{m\in M:\ \pair{x}{m}=0,\ |\pair{v}{m}|\le1 \text{ for all } v\in\V\setminus B\}.
\]
Then $K(B,x)\subseteq\E(P)$.
If $K(B,x)$ contains a $\Z$-basis of the lattice $x^{\perp}\cap M$, then $\E(P)$ contains a $\Z$-basis of $M$, i.e., $P$ is Ewald.
\end{proposition}

\begin{proof}
By Lemma~\ref{lem:pair}, $B$ is $\{x\}$, $\{x,-x\}$, or $\{x,u\}$ with $y=x+u$ a good vertex.
Let $m\in K(B,x)$.
Then $\pair{\pm x}{m}=0$, and in the third case $\pair{u}{m}=\pair{y}{m}-\pair{x}{m}=\pair{y}{m}$ has absolute value at most $1$.
So $m\in\E(P)$.
Let $m_0=x^*_G\in\E(P)$ be the point given by Proposition~\ref{prop:one} or~\ref{prop:nonadjacent}.
Since $\pair{x}{m_0}=1$, we have $M=\Z m_0\oplus(x^\perp\cap M)$, so $m_0$ together with a basis of $x^\perp\cap M$ is a basis of $M$.
\end{proof}

Recall that a transversal of set $S$ is called minimal if it has the smallest possible number of elements, $\ud(S)$, and \emph{inclusion-minimal} if no proper subset is a transversal.
\begin{question}\label{q:split}
Let $B$ be an inclusion-minimal independent transversal with one or two elements, and let $x\in B$.
Does $K(B,x)$ always contain a $\Z$-basis of $x^\perp\cap M$?
\end{question}

The answer is yes for all smooth Fano polytopes of dimension at most $6$: we checked all $18\,276$ pairs $(B,x)$ occurring in the $3\,675$ polytopes of dimension at most $6$ which are not unimodular and have an independent transversal.
Since Ewald's conjecture is known in these dimensions, the interest of Question~\ref{q:split} lies in a proof that works in all dimensions; a positive answer would prove Ewald's conjecture for all smooth Fano polytopes with an independent transversal.
The statement is not a formal consequence of unimodularity: for the unimodular set $\{e_1,e_2\}\subset\Z^2$ and the additional vector $x=(1,3)$, the lattice $x^\perp\cap M$ is generated by $(3,-1)$, which takes the value $3$ at $e_1$.
For two adjacent bad vertices we do not know an analogue of Proposition~\ref{prop:split}.

\subsection{Larger unimodular defect}
By~\eqref{eq:defect-bound} every smooth Fano polytope has finite unimodular defect, so the statement ``\AP{} holds for smooth Fano polytopes of every unimodular defect'' is equivalent to \AP{} for all smooth reflexive polytopes.
Theorem~\ref{thm:A} settles defects $0$, $1$ and $2$.

\begin{question}\label{q:defect3}
Is there a smooth Fano polytope $Q$ with $\ud(Q)\ge3$?
\end{question}

By Table~\ref{tab:defects}, such a polytope has dimension at least $7$.
If the answer is negative, \AP{} holds for all smooth reflexive polytopes by Theorem~\ref{thm:A}.
By Theorem~\ref{thm:B}\ref{C:size}, a minimal transversal with three elements always contains two adjacent vertices.
The method of Proposition~\ref{prop:adjacent} extends as follows.
Let $B$ be a transversal contained in a facet $G$, with $|B|=k$, and let $R=G\setminus B$.
Writing vertices in the basis $G$ and keeping only the $B$-coordinates defines a projection $\pi\colon N\to\Z^k$ with kernel $\Z R$, and for good vertices $y_1,\dots,y_k$ one has
\[
\det\big(\pi(y_1),\dots,\pi(y_k)\big)=\pm\det(y_1,\dots,y_k,R)\in\{0,\pm1\},
\]
because $R$ consists of good vertices.
So the images of the good vertices form a unimodular set in $\Z^k$.
If $Y$ is a set of $k-1$ good vertices such that $R\cup Y$ is a ridge, then $R\cup Y$ is flat if and only if all maximal minors of the $(k-1)\times k$ matrix with rows $\pi(y)$, $y\in Y$, lie in $\{0,\pm1\}$.
For $k=2$ this is the condition $|a|\le1$ or $|b|\le1$ of Proposition~\ref{prop:adjacent}, which we derived from a single determinant; for $k\ge3$ we do not know whether a suitable $Y$ always exists.

\section*{AI declaration and Acknowledgment}

Example \ref{ex:adjacent}, \ref{ex:hexagon} and table \ref{tab:certificates} in section 7, subsection 8.1, 8.3, table \ref{tab:defects} are generated by Claude opus 5.5. GPT sol 6.1 has been used to verify the whole paper and corrects remark 8.2. Besides, the language of the whole paper is polished by GPT sol 6.1.

\end{document}